\documentclass[12pt]{amsart}

\usepackage{latexsym, amsmath, amssymb, amsfonts, amscd, amsthm, mathrsfs}
\newcommand{\Cb}{{\mathbb C}}

\newcommand{\Nb}{{\mathbb N}}

\newcommand{\pa}{\|}

\newtheorem{theorem}{Theorem}[section]

\newtheorem{lemma}[theorem]{Lemma}
\newtheorem{proposition}[theorem]{Proposition}

\theoremstyle{definition}

\begin{document}

\title{A Hilbert $C^*$-module admitting no frames}
\author{Hanfeng Li}
\thanks{Partially supported by NSF Grant DMS-0701414.}
\address{Department of Mathematics \\
SUNY at Buffalo \\
Buffalo, NY 14260-2900, U.S.A.} \email{hfli@math.buffalo.edu}
\date{September 29, 2009}


\begin{abstract}
We show that every infinite-dimensional commutative unital $C^*$-algebra has a Hilbert $C^*$-module
admitting no frames. In particular, this shows that Kasparov's stabilization
theorem for countably generated Hilbert $C^*$-modules can not be extended to arbitrary Hilbert $C^*$-modules.
\end{abstract}

\maketitle

%\tableofcontents

% Things to revise:
%1. line 8 of page 2, "finitely generated" should be "countably generated".
%2. change "\cup" to "\bigcup" at various places.

%%%%%%%%%%%%%%%%%%%%%%%%%%%%%%%%%%%%%%%%%%%%%%%%%%%%%%%%%%%%%%%%%%%%%%%
\section{Introduction} \label{introduction:sec}

Kasparov's celebrated stabilization theorem \cite{Kasparov} says that for any $C^*$-algebra $A$ and any
countably generated (right)
Hilbert $A$-module $X_A$, the direct sum $X_A\oplus H_A$ is isomorphic to $H_A$ as Hilbert
$A$-modules, where $H_A$ denotes the standard Hilbert $A$-module $\oplus_{j\in \Nb}A_A$ (see Section~\ref{proof:sec} below for definition).
This theorem plays an important role in Kasparov's $KK$-theory.
%We refer the reader to \cite{Lance, MT, RW, WO}
%for the basics of Hilbert $C^*$-modules.

There has been some generalization of Kasparov's stabilization theorem to a larger class of Hilbert $C^*$-modules \cite{RT}.
It is natural to ask whether Kasparov's stabilization theorem can be generalized to arbitrary Hilbert $A$-modules
via replacing $H_A$ by $\oplus_{j\in J}A_A$ for some large set $J$ depending on $X_A$. In other words,
given any Hilbert $A$-module $X_A$, is $X_A\oplus (\oplus_{j\in J}A_A)$ is isomorphic to $\oplus_{j\in J}A_A$ as Hilbert
$A$-modules for some set $J$?

An affirmative answer to the above question would imply that $X_A$ is a direct summand of $\oplus_{j\in J}A_A$, i.e,
$X_A\oplus Y_A$ is isomorphic to $\oplus_{j\in J}A_A$ for some Hilbert $A$-module $Y_A$.

In \cite{FL02} Frank and Larson generalized the classical frame theory from Hilbert spaces to the setting of Hilbert $C^*$-modules.
Given a unital $C^*$-algebra $A$ and a Hilbert $A$-module $X_A$, a set $\{x_j: j\in J\}$ of elements in $X_A$
is called
a {\it frame} of $X_A$ \cite[Definition 2.1]{FL02} if there is a real constant $C>0$ such that
$\sum_{j\in J}\left<x, x_j\right>_A \left<x_j, x\right>_A$ converges in the ultraweak operator topology to some element
in the universal enveloping von Neumann algebra $A^{**}$ of $A$ \cite[page 122]{Takesaki} and
\begin{eqnarray} \label{frame0:eq}
C\left<x, x\right>_A\le \sum_{j\in J}\left<x, x_j\right>_A\left<x_j, x\right>_A\le C^{-1}\left<x, x\right>_A
\end{eqnarray}
for every $x\in X_A$.
%\begin{eqnarray} \label{frame0:eq}
%C\varphi(\left<x, x\right>_A)\le \sum_{j\in J}\varphi(\left<x, x_j\right>_A\left<x_j, x\right>_A)\le C^{-1}\varphi(\left<x, x\right>_A)
%\end{eqnarray}
%for all $x\in X_A$ and state $\varphi$ of $A$.
It is called a {\it standard frame} of $X_A$ if furthermore
$\sum_{j\in J}\left<x, x_j\right>_A\left<x_j, x\right>_A$ converges in norm for every $x\in X_A$.
Frank and Larson showed that a Hilbert $A$-module $X_A$ has a standard frame if and only if
$X_A$ is a direct summand of $\oplus_{j\in J}A_A$ for some set $J$ \cite[Example 3.5, Theorems 5.3 and 4.1]{FL02}.
From Kasparov's stabilization theorem they concluded that every countably generated Hilbert $A$-module has a standard frame.
However, the existence of standard frames for general Hilbert $A$-modules was left open.
In fact, even the existence of frames for general Hilbert $A$-modules is open, as Frank and Larson
asked in Problem 8.1 of \cite{FL02}.

The purpose of this note is to show that the answers to these questions are in general negative, even for
every infinite-dimensional commutative unital $C^*$-algebra:

\begin{theorem} \label{main:thm}
Let $A$ be a unital commutative $C^*$-algebra. Then the following are equivalent:
\begin{enumerate}
\item $A$ is finite-dimensional,

\item for every Hilbert $A$-module $X_A$, $X_A\oplus (\oplus_{j\in J}A_A)$ is isomorphic to $\oplus_{j\in J}A_A$ as Hilbert $A$-modules
for some set $J$,

\item for every Hilbert $A$-module $X_A$, $X_A\oplus Y_A$ is  isomorphic to $\oplus_{j\in J}A_A$ as Hilbert $A$-modules
for some set $J$ and some Hilbert $A$-module $Y_A$,

\item every Hilbert $A$-module $X_A$ has a  standard frame,

\item every Hilbert $A$-module $X_A$ has a frame.
\end{enumerate}
\end{theorem}

In Section~\ref{continuous field:sec} we establish some result on continuous fields of Hilbert spaces.
Theorem~\ref{main:thm} is proved
in Section~\ref{proof:sec}.

\noindent{\it Acknowledgements.} I am grateful to Michael Frank, Cristian Ivanescu, David Larson, and Jingbo Xia for helpful comments.

%%%%%%%%%%%%%%%%%%%%%%%%%%%%%%%%%%%%%%%%%%%%%%%%%%%%%%%%%%%%%%%%%%%%%%%%%%%%%%%%%%%%%%%%%%%%%%%%%%%%%%%%%
\section{Continuous fields of Hilbert spaces} \label{continuous field:sec}

In this section we prove Proposition~\ref{field on compact:prop}.

\begin{lemma} \label{uncountable:lemma}
There exists an uncountable set $S$ of injective maps $\Nb\rightarrow \Nb$ such that for any distinct $f,g \in S$, $f(n)\neq g(n)$ for all but finitely many $n\in \Nb$, and $f(n)\neq g(m)$ for all $n\neq m$.
\end{lemma}
\begin{proof} Take an injective map $T$ from $\cup_{n\in \Nb} \Nb^n$ into $\Nb$. For each $x: \Nb\rightarrow \Nb$
define $f_x: \Nb\rightarrow \Nb$ by $f_x(n)=T(x(1), x(2), \dots, x(n))$ for all $n\in \Nb$. If $x\neq y$, say, $x(m)\neq y(m)$ for
some $m\in \Nb$, then $f_x(n)\neq f_y(n)$ for all $n\ge m$. Now the set $S:=\{f_x\in \Nb^{\Nb}: x\in \Nb^{\Nb}\}$ satisfies the requirement.
\end{proof}

We refer the reader to \cite[Chapter 10]{Dixmier} for details on continuous fields of Banach spaces.
Let $T$ be a topological space. Recall that a {\it continuous field of (complex) Banach spaces over $T$} is a family
$(H_t)_{t\in T}$ of complex Banach spaces,  with a set $\Gamma\subseteq \prod_{t\in T}H_t$ of sections such that:
 \begin{enumerate}
 \item[(i)] $\Gamma$ is a linear subspace of $\prod_{t\in T}H_t$,

 \item[(ii)] for every $t\in T$, the set of $x(t)$ for $x\in \Gamma$ is dense in $H_t$,

 \item[(iii)] for every $x\in \Gamma$ the function $t\mapsto \pa x(t)\pa$ is continuous on $T$,

 \item[(iv)] for any $x\in \prod_{t\in T}H_t$, if for every $t\in T$ and every $\varepsilon>0$ there exists
 an $x'\in \Gamma$ with $\pa x'(s)-x(s)\pa<\varepsilon$ for all $s$ in some neighborhood of $t$, then $x\in \Gamma$.
 \end{enumerate}

\begin{lemma} \label{field on interval:lemma}
For each $s\in [0, 1]$ there exists a continuous field of Hilbert spaces $((H_t)_{t\in [0, 1]}, \Gamma)$ over
$[0, 1]$ such that $H_t$ is
separable for every $t\in [0, 1]\setminus \{s\}$ and $H_s$
is nonseparable.
\end{lemma}
\begin{proof} We consider the case $s=0$. The case $s>0$ can be dealt with similarly.
Let $H$ be an infinite-dimensional separable Hilbert space. Take an orthonormal
basis $\{e_n\}_{n\in \Nb}$ of $H$.
Let $S$ be as in Lemma~\ref{uncountable:lemma}.
For every $f\in S$ and every $t\in (0, 1]$
set $v_{f, t}\in H$ by
$$v_{f, t}=\cos(\frac{1/n-t}{1/n-1/(n+1)}\cdot \frac{\pi}{2})e_{f(n)}+\sin(\frac{1/n-t}{1/n-1/(n+1)}\cdot \frac{\pi}{2})e_{f(n+1)}$$
for $1/(n+1)\le t\le 1/n$. Set $H_t$ to be the closed linear span of $\{v_{f, t}: f\in S\}$ in $H$ for each $t\in (0, 1]$.
Let $H_0$ be a Hilbert space with an orthonormal basis $\{e'_f\}_{f\in S}$ indexed by $S$.
Then $H_t$ is separable for each $t\in (0, 1]$ while $H_0$ is nonseparable.

For each $f\in S$, consider the section $x_f\in \prod_{t\in [0, 1]}H_t$ defined by
$x_f(t)=v_{f, t}$ for $t\in (0, 1]$, and $x_f(0)=e'_f$. Then $x_f(t)$ is a unit vector in $H_t$ for every $t\in [0, 1]$, and the map $t\mapsto x_f(t)\in H$ is continuous on $(0, 1]$. Denote by $V$ the linear span of
$\{x_f: f\in S\}$ in $\prod_{t\in [0, 1]}H_t$.

We claim that the function $t\mapsto \pa y(t)\pa$ is continuous on $[0, 1]$ for
every $y\in V$. Let $y\in V$. Say, $y=\sum^m_{j=1}\lambda_jx_{f_j}$ for some pairwise distinct $f_1, \dots, f_m$ in
$S$ and some $\lambda_1, \dots, \lambda_m$ in $\Cb$. Then the map $t\mapsto y(t)=\sum^m_{j=1}\lambda_jx_{f_j}(t)\in H$ is continuous
on $(0, 1]$. Thus the function $t\mapsto \pa y(t)\pa$ is continuous on $(0, 1]$.
When $t$ is small enough, $x_{f_1}(t), \dots, x_{f_m}(t)$ are orthonormal and hence $\pa y(t)\pa=(\sum^{m}_{j=1}|\lambda_j|^2)^{1/2}$.
Thus the function $t\mapsto \pa y(t)\pa$ is also continuous at $t=0$. This proves the claim.

Since $V$ satisfies the conditions (i), (ii), (iii) in the definition of continuous fields of Banach spaces (with $\Gamma$ replaced by $V$),
by \cite[Proposition 10.2.3]{Dixmier} one has the continuous field of Hilbert spaces $((H_t)_{t\in [0, 1]}, \Gamma)$ over $[0, 1]$, where
$\Gamma$ is the set of all sections $x\in \prod_{t\in [0, 1]}H_t$ such that for every $t\in [0, 1]$ and every $\varepsilon>0$ there exists
 an $x'\in V$ with $\pa x'(t')-x(t')\pa<\varepsilon$ for all $t'$ in some neighborhood of $t$.
\end{proof}

\begin{lemma} \label{infinite value:lemma}
Let $Z$ be an infinite compact Hausdorff space. Then there exists a real-valued continuous function
$f$ on $Z$ such that $f(Z)$ is infinite.
\end{lemma}
\begin{proof} Suppose that every real-valued continuous function on $Z$  has finite image.
Take a non-constant real-valued continuous function $h$ on $Z$. Say, $h(Z)=B\cup D$
with both $B$ and $D$ being nonempty finite sets. Then at least one of $h^{-1}(B)$ and
$h^{-1}(D)$ is infinite. Say, $h^{-1}(D)$ is infinite.
Set $W_1=h^{-1}(B)$. Then $W_1$ and $Z\setminus W_1$
are both nonempty closed and open subsets of $Z$, and $Z\setminus W_1$ is infinite.

Since
every real-valued continuous function $g$ on $Z\setminus W_1$ extends to a real-valued continuous function on $Z$,
$g$ must have finite image. Applying the above argument to $Z\setminus W_1$ we can find $W_2\subseteq Z\setminus W_1$ such
that both $W_2$ and $Z\setminus (W_1\cup W_2)$ are nonempty closed and open subsets of $Z\setminus W_1$,  and $Z\setminus (W_1\cup W_2)$ is infinite.
Inductively, we find pairwise disjoint nonempty closed and open subsets $W_1, W_2, W_3, \dots$ of $Z$. Now define
$f$ on $Z$ by $f(z)=1/n$ if $z\in W_n$ for some $n\in \Nb$ and $f(z)=0$ if $z\in Z\setminus \cup^{\infty}_{n=1}W_n$.
Then $f$ is a continuous function on $Z$ and $f(Z)$ is infinite, contradicting our assumption. Therefore there exists
a real-valued continuous function on $Z$ with infinite image.
\end{proof}

\begin{proposition} \label{field on compact:prop}
Let $Z$ be an infinite compact Hausdorff space. Then there exist a continuous field
of Hilbert spaces $((H_z)_{z\in Z}, \Gamma)$ over $Z$, a countable subset $W\subseteq Z$,
and a point $z_{\infty}\in \overline{W}\setminus W$ such that $H_z$ is separable for every $z\in W$
while $H_{z_{\infty}}$ is nonseparable.
\end{proposition}
\begin{proof} By Lemma~\ref{infinite value:lemma} we can find a continuous map $f: Z\rightarrow [0, 1]$ such
that $f(Z)$ is infinite. Note that $f(Z)$ is a compact metrizable space. Thus we can find
a convergent sequence $\{t_n\}_{n\in \Nb}$ in $f(Z)$ such that its limit, denoted by $t_{\infty}$, is not equal to
$t_n$ for any $n\in \Nb$. For each $n\in \Nb$ take $z_n\in f^{-1}(t_n)$. Set $W=\{z_n: n\in \Nb\}$.
Then $W$ is countable and $f(\overline{W})=\overline{\{t_n: n\in \Nb\}}\ni t_{\infty}$. Take
$z_{\infty}\in \overline{W}$ with $f(z_{\infty})=t_{\infty}$. Then $z_{\infty}\not \in W$.

By Lemma~\ref{field on interval:lemma} we can find a continuous field of Hilbert spaces
$((H'_t)_{t\in [0, 1]}, \Gamma')$ over $[0, 1]$ such that $H'_t$ is separable for every $t\in [0, 1]\setminus \{t_{\infty}\}$
while $H'_{t_{\infty}}$ is nonseparable. Set $H_z=H'_{f(z)}$ for each $z\in Z$. Then $H_z$ is separable for
every $z\in W$ while $H_{z_{\infty}}$ is nonseparable. For each $\gamma \in \Gamma'$ set $x_{\gamma}\in \prod_{z\in Z}H_z$
by $x_{\gamma}(z)=\gamma(f(z))\in H'_{f(z)}=H_z$ for all $z\in Z$. Then $V:=\{x_{\gamma}\in \prod_{z\in Z}H_z: \gamma \in \Gamma'\}$ is
a linear subspace of $\prod_{z\in Z}H_z$ satisfying the conditions (i), (ii), (iii)
in the definition of continuous fields of Banach spaces (with $\Gamma$ replaced by $V$).
By \cite[Proposition 10.2.3]{Dixmier} one has the continuous field of Hilbert spaces $((H_z)_{z\in Z}, \Gamma)$ over $Z$, where
$\Gamma$ is the set of all sections $x\in \prod_{z\in Z}H_z$ such that for every $z\in Z$ and every $\varepsilon>0$ there exists
 an $x'\in V$ with $\pa x'(z')-x(z')\pa<\varepsilon$ for all $z'$ in some neighborhood of $z$.
\end{proof}

%%%%%%%%%%%%%%%%%%%%%%%%%%%%%%%%%%%%%%%%%%%%%%%%%%%%%%%%%%%%%%%%%%%%%%%%%%%%%%%%%%%%%
\section{Proof of Theorem~\ref{main:thm}} \label{proof:sec}

In this section we prove Theorem~\ref{main:thm}.

Recall that given a $C^*$-algebra $A$, a {\it (right) Hilbert $A$-module} is a right $A$-module $X_A$ with
an $A$-valued inner product map $\left<\cdot, \cdot\right>_A: X_A\times X_A\rightarrow A$ such that:
\begin{enumerate}
\item[(i)] $\left<\cdot, \cdot\right>_A$ is $\Cb$-linear in the second variable,

\item[(ii)] $\left<x, ya\right>_A=\left<x, y\right>_Aa$ for all $x, y\in X_A$ and $a\in A$,

\item[(iii)] $\left<y, x\right>_A=(\left<x, y\right>_A)^*$ for all $x, y\in X_A$,

\item[(iv)] $\left<x, x\right>_A\ge 0$ in $A$ for every $x\in X_A$, and $\left<x, x\right>_A=0$ only when $x=0$,

\item[(v)] $X_A$ is complete under the norm $\pa x\pa:=\pa \left<x, x\right>_A\pa^{1/2}$.
\end{enumerate}
Two Hilbert $A$-modules are said to be {\it isomorphic} if there is an $A$-module isomorphism between them preserving
the $A$-valued inner products.
We refer the reader to \cite{Lance, MT, RW, WO}
for the basics of Hilbert $C^*$-modules.

We give a characterization of frames avoiding von Neumann algebras.

\begin{proposition} \label{frame:prop}
Let $A$ be a unital $C^*$-algebra and let $X_A$ be a Hilbert $A$-module. Let $\{x_j: j\in J\}$ be a set of elements in $X_A$.
Then $\{x_j: j\in J\}$ is a frame of $X_A$ if and only if there is a real constant $C>0$ such that
\begin{eqnarray} \label{frame:eq}
C\varphi(\left<x, x\right>_A)\le \sum_{j\in J}\varphi(\left<x, x_j\right>_A\left<x_j, x\right>_A)\le C^{-1}\varphi(\left<x, x\right>_A)
\end{eqnarray}
for every $x\in X_A$ and every state $\varphi$ of $A$.
\end{proposition}
\begin{proof}
Suppose that $\{x_j: j\in J\}$ is a frame of $X_A$. Let $C$ be a constant witnessing (\ref{frame0:eq}).
Then every state $\varphi$ of $A$ extends uniquely to a normal state of $A^{**}$, which we still denote by $\varphi$.
Applying $\varphi$ to (\ref{frame0:eq}) we obtain (\ref{frame:eq}). This proves the ``only if" part.

Now suppose that (\ref{frame:eq}) is satisfied for every $x\in X_A$ and every state $\varphi$ of $A$. Let $x\in X_A$.
Note that $\left<x, x_j\right>_A\left<x_j, x\right>_A=(\left<x_j, x\right>_A)^*\left<x_j, x\right>_A\ge 0$
for every $j\in J$.
For any finite subset $F$ of $J$, from (\ref{frame:eq}) we get
$$\varphi(\sum_{j\in F}\left<x, x_j\right>_A\left<x_j, x\right>_A)\le \varphi(C^{-1}\left<x, x\right>_A)$$
for every state $\varphi$ of $A$, and hence
$$ \sum_{j\in F}\left<x, x_j\right>_A\left<x_j, x\right>_A\le C^{-1}\left<x, x\right>_A.$$
Thus the monotone increasing net $\{\sum_{j\in F}\left<x, x_j\right>_A\left<x_j, x\right>_A\}_F$, for $F$ being
finite subsets of $J$ ordered by inclusion, of self-adjoint
elements in $A^{**}$ is bounded above. Represent $A^{**}$ faithful as a von Neumann algebra on
some Hilbert space $H$. Then we may also represent $A^{**}$ naturally as a von Neumann algebra on
the Hilbert space $H^{\infty}=\oplus_{n\in \Nb}H$. By \cite[Lemma 5.1.4]{KR} the above net converges in the weak operator topology
of $B(H^{\infty})$
to some element $a$ of $A^{**}$.
%Since this net is bounded in norm, it converges to $a$ in the ultraweak operator topology.
Since the weak operator topology on $B(H^{\infty})$ restricts to the ultraweak operator topology on $A^{**}$, we see
that the above net converges to $a$ in the ultraweak operator topology.
Then (\ref{frame:eq}) tells us that
$$  C\varphi(\left<x, x\right>_A)\le \varphi(a)\le C^{-1}\varphi(\left<x, x\right>_A)$$
for every normal state $\varphi$ of $A^{**}$. Therefore, $C\left<x, x\right>_A\le a\le C^{-1}\left<x, x\right>_A$ as desired.
This finishes the proof of the ``if" part.
\end{proof}

Let $((H_z)_{z\in Z}, \Gamma)$ be a continuous field of Hilbert spaces over a compact Hausdorff space $Z$.
We shall write the inner product on each $H_z$ as linear in the second variable and conjugate-linear in the first variable.
By \cite[Proposition 10.1.9]{Dixmier} $\Gamma$ is right $C(Z)$-module under the pointwise multiplication, i.e.,
$$ (xa)(z)=x(z)a(z)$$
for all $x\in \Gamma$, $a\in C(Z)$, and $z\in Z$.  By \cite[10.7.1]{Dixmier} for any $x, y\in \Gamma$, the function
$z\mapsto \left<x(z), y(z)\right>$ is in $C(Z)$. From the conditions (iii) and (iv) in the definition of continuous fields of Banach spaces
in Section~\ref{continuous field:sec} one sees that $\Gamma$ is a Banach space under the supremum norm $\pa x\pa:=\sup_{z\in Z}\pa x(z)\pa$.
Therefore $\Gamma$ is a Hilbert $C(Z)$-module with the pointwise $C(Z)$-valued inner product
$$ \left<x, y\right>_{C(Z)}(z)=\left<x(z), y(z)\right>$$
for all $x, y\in \Gamma$, and $z\in Z$. In fact, up to isomorphism every Hilbert $C(Z)$-module arises this way \cite[Theorem 3.12]{Takahashi},
though we won't need this fact except in the case $Z$ is finite.

\begin{lemma} \label{no frame:lemma}
Let $((H_z)_{z\in Z}, \Gamma)$ be a continuous field of Hilbert spaces over a compact Hausdorff space
$Z$. Suppose that there are  a countable subset $W\subseteq Z$
and a point $z_{\infty}\in \overline{W}\setminus W$ such that $H_z$ is separable for every $z\in W$
while $H_{z_{\infty}}$ is nonseparable. Then $\Gamma$ as a Hilbert $C(Z)$-module has no frames.
\end{lemma}
\begin{proof}
Suppose that $\{x_j:j\in J\}$ is a frame of $\Gamma$.
By Proposition~\ref{frame:prop} there is a real constant $C>0$ such that
the inequality (\ref{frame:eq}) holds for
every $x\in X_A$ and every state $\varphi$ of $A$.
For each $z \in Z$ denote by $\varphi_z$ the
state of $C(Z)$ given by evaluation at $z$.
For any $z\in Z$ and any vector $w\in H_z$,
by \cite[Proposition 10.1.10]{Dixmier} we can find $x\in \Gamma$ with $x(z)=w$.
Taking $\varphi=\varphi_z$
in the inequality (\ref{frame:eq}),
we get
\begin{eqnarray} \label{at n:eq}
 C\pa w\pa^2\le \sum_{j\in J}|\left<x_j(z), w\right>|^2\le C^{-1}\pa w\pa^2.
\end{eqnarray}

For each $z\in Z$ let $S_z$ be an orthonormal basis of $H_z$. For
each $w\in S_z$, from (\ref{at n:eq}) we see that the set $F_w:=\{j\in J: \left<x_j(z), w\right>\neq 0\}$
is countable. Note that the set $F_z:=\{j\in J: x_j(z)\neq 0\}$ is exactly $\cup_{w\in S_z}F_w$.
For each $z\in W$, since $H_z$ is separable, $S_z$ is countable and hence $F_z$ is countable.
Then the set $F:=\cup_{z\in W}F_z$ is countable.

Since $H_{z_{\infty}}$ is nonseparable, we can find a unit vector $w\in H_{z_{\infty}}$ orthogonal to $x_j(z_{\infty})$ for
all $j\in F$.
If $j\in J\setminus F$, then $x_j(z)=0$ for all $z\in W$, and hence by the condition (iii) in the definition of continuous fields of Banach spaces
in Section~\ref{continuous field:sec} we conclude that $x_j(z_{\infty})=0$. Therefore
$\left<x_j(z_{\infty}), w\right>=0$ for all $j\in J$, contradicting (\ref{at n:eq}).
Thus $\Gamma$ has no frames.
\end{proof}

For any $C^*$-algebra $A$, $A$ as a right $A$-module is a Hilbert $A$-module with the $A$-valued inner product $\left<a, b\right>_A=a^*b$
for all $a, b\in A$. Given a family $\{X_j\}_{j\in J}$ of Hilbert $A$-modules, their {\it direct sum}, denoted by $\oplus_{j\in J}X_j$, consists
of $(x_j)_{j\in J}$ in $\prod_{j\in J}X_j$ such that $\sum_{j\in J}\left<x_j, x_j\right>_A$ converges in norm, and has the $A$-valued inner product
$\left<(x_j)_{j\in J}, (y_j)_{j\in J}\right>_A:=\sum_{j\in J}\left<x_j, y_j\right>_A$.

We are ready to prove Theorem~\ref{main:thm}.

\begin{proof}[Proof of Theorem~\ref{main:thm}.] (1)$\Rightarrow$(2): Suppose that $A$ is finite-dimensional.
Then $A=C(Z)$ for a finite discrete space $Z$. For each $z\in Z$ denote by $p_z$ the projection in $C(Z)$ with
$p_z(z')=\delta_{z, z'}$ for all $z'\in Z$.
Let $X_A$ be a Hilbert $A$-module. For any $z\in Z$ and any $xp_z, yp_z\in X_Ap_z$, one has
$\left<xp_z, yp_z\right>_A\in Ap_z=\Cb p_z$. Thus $\left<xp_z, yp_z\right>_A=\lambda p_z$ for some $\lambda\in \Cb$.
Set $\left<xp_z, yp_z\right>=\lambda$. Then it is easily checked that $X_Ap_z$ is a Hilbert space under this inner product,
$((X_Ap_z)_{z\in Z}, \prod_{z\in Z}X_Ap_z)$ is a continuous field of Hilbert spaces over $Z$, and
$X_A$ is isomorphic to $  \prod_{z\in Z}X_Ap_z$ as Hilbert $A$-modules.
Take an infinite-dimensional Hilbert space $H$ such that the Hilbert space dimension of $H$ is no less than that
of $X_Ap_z$ for all $z\in Z$. Then $X_Ap_z\oplus H$ is unitary equivalent to $H$ as Hilbert spaces.
It is readily checked that $((H)_{z\in Z}, \prod_{z\in Z}H)$ is a continuous field of Hilbert spaces over $Z$,
and $(\prod_{z\in Z}X_Ap_z)\oplus (\prod_{z\in Z}H)$ is isomorphic to $\prod_{z\in Z}H$ as Hilbert $A$-modules.
Let $J$ be an orthonormal basis of $H$. Then it is easy to see that $\prod_{z\in Z}H$ and $\oplus_{j\in J}A_A$ are
isomorphic as Hilbert $A$-modules.  Therefore $X_A\oplus (\oplus_{j\in J}A_A)$ and $\oplus_{j\in J}A_A$ are isomorphic as Hilbert $A$-modules.
This proves (1)$\Rightarrow$(2).

The implications (2)$\Rightarrow$(3) and (4)$\Rightarrow$(5) are trivial.

The implication (3)$\Rightarrow$(4) was proved in \cite[Example 3.5]{FL02}. For the convenience of the reader,
we indicate the proof briefly here. Suppose that $X_A$ and $Y_A$ are Hilbert $A$-modules and $X_A\oplus Y_A$ is isomorphic
to $\oplus_{j\in J}A_A$ for some set $J$ as Hilbert $A$-modules. We may assume that $X_A\oplus Y_A=\oplus_{j\in J}A_A$.
Denote by $P$ the orthogonal projection $\oplus_{j\in J}A_A\rightarrow X_A$ sending $x+y$ to $x$ for all
$x\in X_A$ and $y\in Y_A$. For each $s\in J$ denote by $e_s$ the vector in $\oplus_{j\in J}A_A$ with coordinate $1_A\delta_{j, s}$
at each $j\in J$. Set $x_j=P(e_j)$ for each $j\in J$. For any $x\in X_A$, say, $x=\sum_{j\in J}e_ja_j$ with $a_j\in A$ for
each $j\in J$, one has
%\begin{eqnarray*}
% x&=&PPx=P(\sum_{j\in J}e_j\left<e_j, Px\right>_A)\\
% &=&\sum_{j\in J}(Pe_j)\left<Pe_j, x\right>_A=\sum_{j\in J}x_j\left<x_j, x\right>_A,
%\end{eqnarray*}
%and hence $\left<x, x\right>_A=\left<x, \sum_{j\in J}x_j\left<x_j, x\right>_A\right>_A=\sum_{j\in J}\left<x, x_j\right>_A\left<x_j, x\right>_A$
\begin{eqnarray*}
\left<x, x\right>_A&=&\sum_{j\in J}a^*_ja_j=\sum_{j\in J}\left<x, e_j\right>_A\left<e_j, x\right>_A
=\sum_{j\in J}\left<Px, e_j\right>_A\left<e_j, Px\right>_A\\
&=&\sum_{j\in J}\left<x, Pe_j\right>_A\left<Pe_j, x\right>_A
=\sum_{j\in J}\left<x, x_j\right>_A\left<x_j, x\right>_A.
\end{eqnarray*}
%for every $x\in X_A$.
Therefore $\{x_j: j\in J\}$ is a standard frame of $X_A$. This proves (3)$\Rightarrow$(4).

The implication (5)$\Rightarrow$(1) follows from Proposition~\ref{field on compact:prop} and Lemma~\ref{no frame:lemma}.
\end{proof}

%%%%%%%%%%%%%%%%%%%%%%%%%%%%%%%%%%%%%%%%%%%%%%%%%%%%%%%%%%%%%%%%%%%%%%%

\end{document}